\pdfoutput=1 

\documentclass[11pt]{amsart}

\usepackage{amssymb,amsmath,amsthm,enumitem,colonequals,mlmodern,tikz-cd,microtype}
\usepackage[cal=euler,bfcal,bb=px,bfbb]{mathalpha}

\usepackage[top=3.75cm, bottom=3cm, left=3.5cm, right=3.5cm]{geometry}

\usepackage{xcolor}
\colorlet{darkblue}{blue!55!black}
\colorlet{darkcyan}{cyan!50!black}
\colorlet{darkgreen}{green!60!black}

\PassOptionsToPackage{hyphens}{url}
\usepackage{hyperref}
\hypersetup{
    colorlinks=true,
    linkcolor=darkblue,
    urlcolor=darkcyan,
    citecolor=darkgreen,
}

\def\eqref#1{\textcolor{darkblue}{(\ref{#1})}}

\usepackage[nameinlink]{cleveref} 
\Crefformat{section}{#2\S#1#3}
\Crefmultiformat{section}{#2\S\S#1#3}{ and~#2#1#3}{, #2#1#3}{, and~#2#1#3}

\usepackage[pagewise]{lineno}
\let\oldequation\equation
\let\oldendequation\endequation
\renewenvironment{equation}{\linenomathNonumbers\oldequation}{\oldendequation\endlinenomath}
\expandafter\let\expandafter\oldequationstar\csname equation*\endcsname
\expandafter\let\expandafter\oldendequationstar\csname endequation*\endcsname
\renewenvironment{equation*}{\linenomathNonumbers\oldequationstar}{\oldendequationstar\endlinenomath}
\let\oldalign\align
\let\oldendalign\endalign
\renewenvironment{align}{\linenomathNonumbers\oldalign}{\oldendalign\endlinenomath}
\expandafter\let\expandafter\oldalignstar\csname align*\endcsname
\expandafter\let\expandafter\oldendalignstar\csname endalign*\endcsname
\renewenvironment{align*}{\linenomathNonumbers\oldalignstar}{\oldendalignstar\endlinenomath}

\theoremstyle{plain}
\newtheorem{theorem}{Theorem}[section]
\newtheorem{lemma}[theorem]{Lemma}
\newtheorem{corollary}[theorem]{Corollary}
\newtheorem{proposition}[theorem]{Proposition}
\newtheorem*{thm}{Theorem}

\theoremstyle{definition}
\newtheorem{definition}[theorem]{Definition}
\newtheorem{example}[theorem]{Example}
\newtheorem{remark}[theorem]{Remark}

\newtheorem*{ack}{Acknowledgment}

\AddToHook{env/conjecture/begin}{\crefalias{theorem}{conjecture}}
\AddToHook{env/lemma/begin}{\crefalias{theorem}{lemma}}
\AddToHook{env/corollary/begin}{\crefalias{theorem}{corollary}}
\AddToHook{env/proposition/begin}{\crefalias{theorem}{proposition}}
\AddToHook{env/definition/begin}{\crefalias{theorem}{definition}}
\AddToHook{env/remark/begin}{\crefalias{theorem}{remark}}
\AddToHook{env/setup/begin}{\crefalias{theorem}{setup}}
\AddToHook{env/example/begin}{\crefalias{theorem}{example}}
\AddToHook{env/conjecture/begin}{\crefalias{theorem}{Conjecture}}

\numberwithin{equation}{section}
\numberwithin{theorem}{section}

\title[Singly compactly generated $t$-structures for schemes]{Nonexistence of singly compactly generated $t$-structures for schemes}

\author[A.~Bhaduri]{Anirban Bhaduri}
\address{A.~Bhaduri,
Department of Mathematics,
University of South Carolina, 
Columbia, SC, U.S.A.}
\email{abhaduri@email.sc.edu}

\author[T.~De Deyn]{Timothy De Deyn}
\address{T.~De Deyn,
School of Mathematics and Statistics,
University of Glasgow, 
Glasgow G12 8QQ,
United Kingdom}
\email{timothy.dedeyn@glasgow.ac.uk}

\author[M.~Hrbek]{Michal Hrbek}
\address{M.~Hrbek,
Institute of Mathematics of the Czech Academy of Sciences,
Prague, Czechia}
\email{hrbek@math.cas.cz}

\author[P.~Lank]{Pat Lank}
\address{P.~Lank,
Dipartimento di Matematica “F. Enriques”, Universit\`{a} degli Studi di Milano, Milano, Italy}
\email{plankmathematics@gmail.com}

\author[K.~Manali-Rahul]{Kabeer Manali-Rahul}
\address{K.~Manali-Rahul,
Max Planck Institute for Mathematics,
Bonn, Germany}
\email{kabeermr.maths@gmail.com}

\date{\today}



\begin{document}
    
\begin{abstract}
    We show the first instances of schemes whose standard aisles on their derived category of quasi-coherent sheaves are not singly compactly generated.
\end{abstract}

\maketitle


\section{Introduction}
\label{sec:intro}

To date, the territories of derived categories associated with schemes remain largely unexplored. Researchers in these lands now possess a variety of tools at their disposal. Among these is a notion introduced by \cite{Beilinson/Berstein/Deligne/Gabber:2018} called a \textit{$t$-structure}. Roughly speaking, a $t$-structure on a triangulated category $\mathcal{T}$ is a pair of subcategories $(\mathcal{T}^{\leq 0}, \mathcal{T}^{\geq 0})$ satisfying axioms that describe objects of $\mathcal{T}$ in terms of those in $\mathcal{T}^{\leq 0}$ and $\mathcal{T}^{\geq 0}$. In fact, the $t$-structure is completely determined by its \textit{aisle} $\mathcal{T}^{\leq 0}$ \cite{Keller/Vossieck:1988}. These tools provide a coarser, and hence more digestible, topography of these lands.

It remains unknown whether the standard $t$-structure is singly compactly generated, i.e.\ generated by a single compact generator, for any Noetherian scheme. Thus far, by \cite[Theorem 3.2]{Neeman:2022}, all that is known is that the standard aisle $D^{\leq 0}_{\operatorname{qc}}(X)$ is always `equivalent' (in the sense of \cite[Definition 0.18]{Neeman:2025}) to a singly compactly generated aisle. To some, there was an expectation that the standard aisle was itself singly compactly generated. 

However, it turns out, this is \textbf{horridly wrong}:

\begin{thm}
    [see \Cref{thm:coarse_moduli}]
    This is false for $\mathbb{P}^1_k$ over a field $k$. In fact, it fails more generally for proper tame Deligne--Mumford stacks $\mathcal{X}$ of positive Krull dimension over $k$.
\end{thm}

To find counterexamples is not an obvious, nor straightforward, task. Particularly, for the case of $\mathbb{P}^1_k$ where $k=\mathbb{C}$, we checked the failure by using stability of vector bundles (see \Cref{sec:appendix_semistable}). However, the main proof in the text does not require stability, and holds far more generally. Ultimately, we show that $D^{\leq 0}_{\operatorname{qc}}(\mathcal{X})$ being singly compactly generated implies every object of $\operatorname{coh}(\mathcal{X})$ is Artinian (see \Cref{cor:abelian_singly_compactly_generated_implies_Artinian}). In fact, singly compactly generated is equivalent to $\operatorname{coh}(\mathcal{X})$ admitting a weak generator (see \Cref{lem:compact_generator_to_weak}). And for the case of proper algebraic spaces over a field, every coherent sheaf being Artinian occurs precisely when the algebraic space is an affine Artinian scheme (see \Cref{prop:proper_case}). Lastly, we show Noetherian schemes (more generally algebraic spaces) which are proper and of positive relative dimension over a base cannot have singly compactly generated standard aisles (see \Cref{prop:positive_relative_dimension}). For example, this includes proper schemes over the integers with positive dimensional fibers, which gives `arithmetic' cases.

\begin{ack}
    De Deyn was supported by ERC Consolidator Grant 101001227 (MMiMMa). Hrbek was supported by the GA\v{C}R project 23-05148S and the Czech Academy of Sciences (RVO 67985840). Lank was supported under the ERC Advanced Grant 101095900-TriCatApp. Manali-Rahul is grateful to Max Planck Institute for Mathematics in Bonn for its hospitality and financial support. The authors thank Leovigildo Alonso Tarr\'{i}o, Uttaran Dutta, Ana Jerem\'{i}as L\'{o}pez, Amnon Neeman, and Fei Peng for discussions and comments.
\end{ack}

\section{Preliminaries}
\label{sec:prelim}

\subsection{Abelian categories}
\label{sec:prelim_abelian}

Let $\mathcal{A}$ be an essentially small abelian category. Recall that an object $E\in \mathcal{A}$ is called \textbf{Artinian} (resp.\ \textbf{Noetherian}) if any descending (resp.\ ascending) chain of subobjects of $E$ becomes stationary. We say $\mathcal{A}$ is \textbf{Artinian} (resp.\ \textbf{Noetherian}) if every object is such. In the case $\mathcal{A}$ is both Artinian and Noetherian, we say it is a \textbf{length category}. When $\mathcal{A}$ is additionally Grothendieck, it is called \textbf{locally Noetherian} when every object of $\mathcal{A}$ is a directed union of its Noetherian subobjects, which we denote by $\operatorname{noeth}(\mathcal{A})$.
Furthermore, recall an object $G\in \mathcal{A}$ is called a \textbf{weak generator} if $\operatorname{Hom}(G,-)\colon \mathcal{A} \to \operatorname{Ab}$ is faithful on objects (i.e.\ $\operatorname{Hom}(G,E)\cong 0$ implies $E\cong 0$). Lastly, if $R$ is a commutative Noetherian ring, we say $\mathcal{A}$ is \textbf{$\operatorname{Hom}$-finite} over $R$ if it is an $R$-linear category and $\operatorname{Hom}(A,B)$ is a finite length $R$-module for all $A,B\in \mathcal{A}$. This notion of $\operatorname{Hom}$-finiteness may differ from others in the literature; e.g.\ finitely generated module as opposed to finite length. 
We sometimes abusively say `$\operatorname{Hom}$-finite' leaving the commutative ring $R$ implicit.

\subsection{\texorpdfstring{$t$}{t}-structures}
\label{sec:prelim_t-structures}

Let $\mathcal{T}$ be a compactly generated triangulated category. We discuss $t$-structures on triangulated categories. See \cite{Keller/Vossieck:1988,Beilinson/Berstein/Deligne/Gabber:2018} for details. Denote its subcategory of compact objects by $\mathcal{T}^c$. A pair of strictly full subcategories $(\mathcal{T}^{\leq 0}, \mathcal{T}^{\geq 0})$ of $\mathcal{T}$ is a \textbf{$t$-structure} if  $\operatorname{Hom}(A,B) = 0$ for all $A \in \mathcal{T}^{\leq 0}$ and $B \in \mathcal{T}^{\geq 0}[-1]$; $\mathcal{T}^{\leq 0}[1] \subseteq \mathcal{T}^{\leq 0}$ and $\mathcal{T}^{\geq 0}[-1] \subseteq \mathcal{T}^{\geq 0}$; and for every $E \in \mathcal{T}$, there is a distinguished triangle
\begin{displaymath}
    \tau^{\leq 0} E \to E \to \tau^{\geq 1} E \to (\tau^{\leq 0} E)[1]
\end{displaymath}
with $\tau^{\leq 0} E \in \mathcal{T}^{\leq 0}$ and $\tau^{\geq 1} E \in \mathcal{T}^{\geq 0}[-1]$. Also, we say a strictly full subcategory $\mathcal{U} \subseteq \mathcal{T}$ is an \textbf{aisle} if the inclusion $\mathcal{U} \to \mathcal{T}$ admits a right adjoint while $\mathcal{U}$ is closed under positive shifts and extensions. In fact, for any $t$-structure $(\mathcal{T}^{\leq 0}, \mathcal{T}^{\geq 0})$, $\mathcal{U}$ is an aisle. Particularly, any aisle $\mathcal{U}$ determines a $t$-structure $(\mathcal{U}, \mathcal{U}^\perp[1])$ where
\begin{displaymath}
    \mathcal{U}^\perp := \{ T \in \mathcal{T} \mid \forall U \in \mathcal{U}, \operatorname{Hom}(U,T) = 0  \}.
\end{displaymath}
As an example, for an abelian category $\mathcal{A}$, we have the `standard' $t$-structure whose aisle is given by $D^{\leq 0}(\mathcal{A})$ given by objects whose $i$-th cohomology is zero for $i>0$.

Given a set $\mathcal{S}\subseteq \mathcal{T}$, $\overline{\langle \mathcal{S} \rangle}^{(-\infty, 0]}$ is defined to be the smallest cocomplete aisle containing $\mathcal{S}$. By \cite[Theorem 2.3]{Neeman:2021}, such aisles always exist. An aisle $\mathcal{U}$ on $\mathcal{T}$ is \textbf{compactly generated} when there exists a collection of compact objects $\mathcal{P} \subseteq \mathcal{T}^c$ satisfying $\overline{\langle \mathcal{P} \rangle}^{(-\infty, 0]} = \mathcal{U}$. Hence, we say a $t$-structure is \textbf{compactly generated} if its aisle is such. 

\subsection{Algebraic stacks}
\label{sec:prelim_stacks}

Our conventions for algebraic stacks is \cite{stacks-project}. However, we follow \cite[\S1]{Hall/Rydh:2017} for the derived pullback/pushforward adjunction (loc.\ cit.\ follows \cite{Olsson:2007,Laszlo/Olsson:2008a,Laszlo/Olsson:2008b}). Also, we omit the `qc' in the notation of the derived functors compared to \cite{Hall/Rydh:2017}. Symbols $X$, $Y$, etc.\ refer to schemes/algebraic spaces, whereas $\mathcal X$, $\mathcal Y$, etc.\ refer to algebraic stacks. Let $\mathcal{X}$ be a Noetherian algebraic stack. 

Denote by $\operatorname{Mod}(\mathcal{X})$ the Grothendieck abelian category of sheaves of $\mathcal{O}_\mathcal{X}$-modules on the lisse-\'{e}tale site of $\mathcal{X}$ and $\operatorname{Qcoh}(\mathcal{X})$ (resp.\ $\operatorname{coh}(\mathcal{X})$) for  the strictly full subcategory of $\operatorname{Mod}(\mathcal{X})$ consisting of quasi-coherent (resp.\ coherent) sheaves. Define $D(\mathcal{X}):=D(\operatorname{Mod}(\mathcal{X}))$ for the (unbounded) derived category of $\operatorname{Mod}(\mathcal{X})$. Set $D_{\operatorname{qc}}(\mathcal{X})$ (resp.\ $D^b_{\operatorname{coh}}(\mathcal{X})$) for the full subcategory of $D(\mathcal{X})$ consisting of complexes with quasi-coherent cohomology sheaves (resp.\ which are bounded and have coherent cohomology sheaves). Moreover, $\operatorname{Perf}(\mathcal{X})$ is the full subcategory of perfect complexes in $D_{\operatorname{qc}}(\mathcal{X})$ which can be defined for any ringed site \cite[\href{https://stacks.math.columbia.edu/tag/08G4}{Tag 08G4}]{stacks-project}, and so, in particular, for $\mathcal{X}$ by looking at its lisse-\'{e}tale site. As a warning, the compact objects of $D_{\operatorname{qc}}(\mathcal{X})$ are perfect complexes \cite[Lemma 4.4]{Hall/Rydh:2017}, although the converse need not be true. 

We say $\mathcal{X}$ is \textbf{affine-pointed} if every morphism $\operatorname{Spec}(k) \to \mathcal{X}$ from a field $k$ is affine; this is automatic when $\mathcal X$ has quasi-affine or quasi-finite diagonal.
Furthermore, $\mathcal{X}$ is said to satisfy the \textbf{Thomason condition} if there is a cardinal $\beta$ such that $D_{\operatorname{qc}}(\mathcal{X})$ is compactly generated by a collection of cardinalilty at most $\beta$ and for each closed subset $Z$ of $|\mathcal{X}|$ with quasi-compact complement, there is a perfect complex $P$ with support $Z$. 

\section{Results}
\label{sec:results}

\subsection{Locally Noetherian Abelian categories}
\label{sec:results_abelian}

\begin{definition}
    \label{def:approx}
    Let $\mathcal{A}$ be a locally Noetherian Grothendieck abelian category. We say $\mathcal{A}$ satisfies \textbf{approximation by compacts} for every object $E\in D^-(\mathcal{A})$ with Noetherian cohomology and for every $m\in \mathbb{Z}$ there exists a compact $P \in D(\mathcal{A})$ and a morphism $P \to E$ inducing isomorphisms $H^i(C) \to H^i(E)$ for $i > m$ and a surjection for $i = m$.
\end{definition}

\begin{example}
    Approximation by compacts of $\mathcal{A} = \operatorname{Qcoh}(-)$ for a quasi-compact quasi-separated scheme was initially due to Lipman--Neeman for schemes \cite{Lipman/Neeman:2007}, which was later extended to algebraic spaces \cite[\href{https://stacks.math.columbia.edu/tag/08HH}{Tag 08HH}]{stacks-project}. There are also algebraic stacks satisfying this, but a few extra conditions and facts are needed. If $\mathcal{X}$ is a quasi-compact quasi-separated algebraic stack with affine diagonal, then \cite[Theorem 1.2 and proof of Lemma 2.5]{Hall/Neeman/Rydh:2019} tells us the natural functor $D(\operatorname{Qcoh}(\mathcal{X})) \to D_{\operatorname{qc}}(\mathcal{X})$ is an equivalence which respects the standard $t$-structures. 
    Additionally, in such cases, if $\mathcal{X}$ has quasi-finite separated diagonal or is Deligne--Mumford of characteristic zero, then \cite[Theorem A]{Hall/Lamarche/Lank/Peng:2025} tells us $D_{\operatorname{qc}}(\mathcal{X})$, and hence $D(\operatorname{Qcoh}(\mathcal{X}))$, satisfies approximation by compacts.
    Details regarding $\operatorname{Qcoh}(-)$ being a locally Noetherian Grothendieck abelian category are given in the proofs of \Cref{prop:proper_case} and \Cref{thm:coarse_moduli}.
\end{example}

\begin{lemma}
    \label{lem:coh_noeth}
    Let $\mathcal{A}$ be a locally Noetherian Grothendieck abelian category.
    Then any compact object in $D(\mathcal{A})$ has bounded and Noetherian cohomology.
\end{lemma}

\begin{proof}
    This is \cite[Proposition 2.10]{Hrbek/Pavon:2023} combined with \cite[Lemma 9.3.7]{Krause:2022}.
\end{proof}

The second part of the following lemma, that requires approximation by compacts, will not be used later in this text; so can be freely be skipped.
Only the first part, that a single compact generator gives a weak generator of the heart will be used.
We have chosen to keep the second part included for interests sake.

\begin{lemma}
    \label{lem:compact_generator_to_weak}
    Let $\mathcal{A}$ be a locally Noetherian Grothendieck abelian category. If $D^{\leq 0}(\mathcal{A})$ is singly compactly generated, then $\operatorname{noeth}(\mathcal{A})$ admits a weak generator $G$. Furthermore, if $\mathcal{A}$ satisfies approximation by compacts, then the converse holds; in particular, $G\in \operatorname{noeth}(\mathcal{A})$ is a weak generator if, and only if, $\overline{\langle G \rangle}^{(-\infty,0]} = D^{\leq 0}(\mathcal{A})$.
\end{lemma}

\begin{proof}
    To start, we show that if $D^{\leq 0}(\mathcal{A})$ is singly compactly generated by $P$, that $G:=H^0 (P)$ is a weak generator for $\operatorname{noeth}(\mathcal{A})$. We know $G$ is Noetherian by \Cref{lem:coh_noeth}, so let $E \in \operatorname{noeth}(\mathcal{A})$ be nonzero. As $P$ compactly generates $D^{\leq 0}(\mathcal{A})$, there is a nonzero morphism $P\to E$. Indeed, as $P\in D^{\leq 0}(\mathcal{A})$ and $E\in \operatorname{noeth}(\mathcal{A})\subseteq D^{\leq 0}(\mathcal{A}) \cap D^{\geq 0}(\mathcal{A})$, there is $n\geq 0$ such that $\operatorname{Hom}(P[n],E)\not=0$. However, if $n>0$, then we have a contradiction as $E[-n]\in D^{\geq 0}(\mathcal{A})[-1]$.
    Now, taking cohomology, we have a nonzero morphism $H^0 (P) \to E$, which implies $G$ is a weak generator as desired. 
    
    Next, in the case $\mathcal{A}$ satisfies approximation by compacts, we show the converse. Assume there is a weak generator $G\in \operatorname{noeth}(\mathcal{A})$. Clearly, $D^{\geq 1}(\mathcal{A}) \subseteq (\overline{\langle G \rangle}^{(-\infty,0]})^\perp $, so we need to check the reverse containment. Choose $E\in (\overline{\langle G \rangle}^{(-\infty,0]})^\perp$ such that $H^i (E)\not=0$ for some $i\leq 0$.
    Now, we may represent $E$ by complex of objects in $\mathcal{A}$, and so, its $i$-th cycle $Z^i (E)\in \mathcal{A}$.
    Moreover, as $\mathcal{A}$ is locally Noetherian, there is an $E^\prime \in \operatorname{noeth}(\mathcal{A})$ and a morphism $E^\prime \to Z^i (E)$ such that the composition $E^\prime \to Z^i (E) \to H^i (E)$ is nonzero. 
    
    We claim there is a nonzero morphism $G[i] \to E$, which would give a contradiction to the fact that $E\in (\overline{\langle G \rangle}^{(-\infty,0]})^\perp$. As $G$ is a weak generator and $H^i (E)$ is nonzero, we can find a nonzero morphism $G \to Z^i (E)$. If $Z^i (E) [i] \to E$ is the inclusion of the cocycle sheaf viewed as a complex, then we have a nonzero morphism $G[i] \to Z^i (E) [i] \to E$ because its $i$-th cohomology is the nonzero morphism $G \to Z^i (E)$.

    So far, we have shown that $D^{\leq 0}(\mathcal{A}) = \overline{\langle G \rangle}^{(-\infty,0]}$ for some $G\in \operatorname{noeth}(\mathcal{A})$. Using that $\mathcal{A}$ satisifies approximation by compacts, there is a perfect complex $P$ and a morphism $P\to G$ which induces an isomorphism on cohomology in degrees $\geq 0$. This ensures $P\in D^{\leq 0}(\mathcal{A})$. Now, if we argue in a similar fashion above with $P$, we can show that $D^{\leq 0}(\mathcal{A}) = \overline{\langle P \rangle}^{(-\infty,0]}$ as desired.

    Lastly, we check that $G\in \operatorname{noeth}(\mathcal{A})$ is a weak generator if, and only if, $\overline{\langle G \rangle}^{(-\infty,0]} = D^{\leq 0}(\mathcal{A})$. If $G\in \operatorname{noeth}(\mathcal{A})$ is a weak generator, then we can argue as above to show $\overline{\langle G \rangle}^{(-\infty,0]} = D^{\leq 0}(\mathcal{A})$. Conversely, if $\overline{\langle G \rangle}^{(-\infty,0]} = D^{\leq 0}(\mathcal{A})$, then \cite[Lemma 3.1]{AlonsoTarrio/Lopez/Salorio:2003} can be used to check $G$ is a weak generator for $\operatorname{noeth}(\mathcal{A})$.
\end{proof}

\begin{lemma}
    \label{lem:weak_geneator_implies_Artinian_category}
    Let $\mathcal{A}$ be Noetherian $\operatorname{Hom}$-finite abelian category which admits intersections on decreasing sequences of subobjects. If $\mathcal{A}$ has a weak generator $G$, then $\mathcal{A}$ is Artinian (in particular, is a length category).
\end{lemma}

\begin{proof}
    While our inspiration for the proof is \cite[Lemma 2.4]{Paquette:2018}, there are differences, so we spell out the proof. Assume the contrary; that is, there is an $M\in \mathcal{A}$ and a strictly decreasing sequence of (necessarily nonzero) subobjects $M = M_0 \supsetneq M_1 \supsetneq \cdots$. Set $M_\infty := \cap^\infty_{n=0} M_n$, and define $N_i:= M_i/M_\infty$. Then $N:=N_0$ is an object of $\mathcal{A}$ with an infinite strictly decreasing chain $N_i$ of nonzero subobjects whose intersection is the zero object. 
    The hypothesis on $\mathcal{A}$ implies that $\operatorname{Hom}(M,N)$ is an Artinian $R$-module.
    Consequently, as $\cap_i \operatorname{Hom}(M,N_i)=0$ there exists an $i$ with $\operatorname{Hom}(M,N_i)$ zero.
    Since $M$ is a weak generator this implies $N_i$ is the zero, giving us a contradiction. This completes the proof.
\end{proof}

\begin{corollary}
    \label{cor:abelian_singly_compactly_generated_implies_Artinian}
    Let $\mathcal{A}$ be a locally Noetherian Grothendieck abelian such that $\operatorname{noeth}(\mathcal{A})$ is $\operatorname{Hom}$-finite. If $D^{\leq 0} (\mathcal{A})$ is singly compactly generated, then $\operatorname{noeth}(\mathcal{A})$ is a length category. 
\end{corollary}

\begin{proof}
    This is immediate from \Cref{lem:compact_generator_to_weak,lem:weak_geneator_implies_Artinian_category}.
\end{proof}

\subsection{Proper over a field}
\label{sec:results_proper_over_field}

In the case of proper schemes over a field, \Cref{cor:abelian_singly_compactly_generated_implies_Artinian} implies the scheme must be Artinian if the standard aisle were singly compactly generated. However, we show this more generally for Noetherian algebraic stacks.

\begin{proposition}
    \label{prop:proper_case}
    Let $X$ be an algebraic space proper over a field. Then $X$ is an (affine) Artinian scheme if, and only if, $D^{\leq 0}_{\operatorname{qc}}(X)$ is singly compactly generated.
\end{proposition}

\begin{proof}
    That $X$ an affine scheme implies $D^{\leq 0}_{\operatorname{qc}}(X)$ is singly compactly generated follows from \cite[\S 2.3]{Canonaco/Neeman/Stellari:2025}. So, we only need to show the converse. 
    First note $D(\operatorname{Qcoh}(X)) \cong D_{\operatorname{qc}}(X)$  by \cite[\href{https://stacks.math.columbia.edu/tag/08H1}{Tag 08H1}]{stacks-project} and that $\operatorname{Qcoh}(X)$ is locally Noetherian by \cite[\href{https://stacks.math.columbia.edu/tag/07UV}{Tag 07UV} \& \href{https://stacks.math.columbia.edu/tag/07UJ}{Tag 07UJ}]{stacks-project}.
    Moreover, As $X$ is proper over $k$, $\operatorname{coh}(X)$ is $\operatorname{Hom}$-finite; see e.g.\ \cite[\href{https://stacks.math.columbia.edu/tag/0D0T}{Tag 0D0T}]{stacks-project}. 
    Consequently, by \Cref{lem:weak_geneator_implies_Artinian_category}, $\operatorname{coh}(X)$ is a length category. Particularly, the structure sheaf $\mathcal{O}_X$ is an Artinian object of $\operatorname{coh}(X)$. Consequently, any increasing chain of closed subschemes of $X$ stabilizes, which tells us $|X|$ is Artinian. Hence, $X$ is affine by \cite[\href{https://stacks.math.columbia.edu/tag/06L3}{Tag 06LZ} \& \href{https://stacks.math.columbia.edu/tag/0ACA}{Tag 0ACA}]{stacks-project}.
\end{proof}

\begin{theorem}
    \label{thm:coarse_moduli}
    Let $\mathcal{X}$ be a tame Deligne--Mumford stack of positive Krull dimension. If $\mathcal{X}$ is proper over a field, then $D^{\leq 0}_{\operatorname{qc}}(\mathcal{X})$ cannot be singly compactly generated. 
\end{theorem}

\begin{proof}
    Assume the contrary, that is, $D^{\leq 0}_{\operatorname{qc}}(\mathcal{X})$ is singly compactly generated. 
    To start, note that \cite[Theorem 1.2]{Hall/Neeman/Rydh:2019} tells us $D(\operatorname{Qcoh}(X)) \cong D_{\operatorname{qc}}(X)$ as $\mathcal{X}$ has affine diagonal. 
    Furthermore, $\operatorname{Qcoh}(X)$ is a locally Noetherian by \cite[Proposition 15.4]{Laumon/Moret-Bailly:2000} (or \cite[\href{https://stacks.math.columbia.edu/tag/0GRE}{Tag 0GRE}]{stacks-project}) and the fact that coherent sheaves are Noetherian (this can be proven by taking a flat cover).
    Moreover, the argument for showing $\operatorname{coh}$ is $\operatorname{Hom}$-finite for a proper scheme/space over a field applies verbatim to algebraic stacks.
    Thus, by \Cref{cor:abelian_singly_compactly_generated_implies_Artinian}, $\operatorname{coh}(\mathcal{X})$ is an Artinian category.
    
    By \cite[Theorem 11.1.2 \& Proposition 11.3.4]{Olsson:2016}, there is a proper morphism $\pi\colon \mathcal{X} \to X$ with $X$ an algebraic space satisfying two key properties. 
    First, $\pi_\ast$ is an exact functor on quasi-coherent sheaves. Second, the unit of the (underived) pull/push adjunction $1\to \pi_\ast \pi^\ast$ is an isomorphism on quasi-coherent sheaves (cf.\ \cite[Proposition 4.5]{Alper:2013}). 
    Moreover, as $\mathcal{X}$ is proper over $k$ and $\pi$ is a proper morphism, it follows that $X$ is proper over $k$ (use \cite[\href{https://stacks.math.columbia.edu/tag/0CQK}{Tag 0CQK}]{stacks-project} as $\pi$ is surjective).

    Now, we finish the proof by showing $\operatorname{coh}(X)$ is Artinian. Let $E\in \operatorname{coh}(X)$. Consider a decreasing sequence of subobjects $\cdots \subseteq E_j \subseteq \cdots \subseteq E_0 =: E$. It need not be the case that $\{\pi^\ast E_j\}$ forms a decreasing sequence in $\pi^\ast E$. However, the sequence $\{A_j:= \operatorname{im}(\pi^\ast (E_j \to E))\}$ does. Thus, as $\operatorname{coh}(\mathcal{X})$ is Artinian, we can find a $j\geq 1$ such that $A_i = A_j$ for all $i\geq j$. Applying $\pi_\ast$, it follows that $E_i = E_j$  for all $i\geq j$. Indeed, exactness of $\pi_\ast$ and the unit being an isomorphism tells us for each $k$,
    \begin{displaymath}
        \pi_\ast (\operatorname{im}(\pi^\ast (E_k \to E))) = \operatorname{im} (\pi_\ast \pi^\ast (E_k \to E)) = \operatorname{im}(E_k \to E) = E_k.
    \end{displaymath}
    Consequently, $\operatorname{coh}(X)$ must be Artinian, and so, \Cref{prop:proper_case} tells us $X$ is an Artinian scheme. However, this is absurd as $\mathcal{X}$, and hence $X$, has positive Krull dimension    (e.g.\ use that $\pi$ is a homeomorphism on topological spaces).
\end{proof}

\begin{remark}
    There is a generalization possible to \Cref{thm:coarse_moduli}.
    Suppose $\mathcal{X}$ is a (potentially non-DM) algebraic stack proper over a field which allows for a separated good moduli space in the sense of \cite{Alper:2013}.
    Then the same proof  as above works, although only for $D(\operatorname{Qcoh}(\mathcal{X}))$ without extra conditions on the algebraic stack to ensure $D(\operatorname{Qcoh}(\mathcal{X}))=D_{\operatorname{qc}}(\mathcal{X})$.
    Moreover, note that the separatedness of the good moduli morphisms is not as readily deducible as for coarse moduli \cite{Alper/Halpern-Leistner/Heinloth:2023}.
\end{remark}

\subsection{Proper over general base}
\label{sec:results_proper_general}

Next, we show how to use the above to obtain statements over an arbitrary base. We start with a lemma.

\begin{lemma}
    \label{lem:affine_base_change}
    Let $f\colon Y \to X$ be an affine morphism of Noetherian algebraic spaces. If $D^{\operatorname{\leq 0}}_{\operatorname{qc}}(X)$ is singly compactly generated, then so is $D^{\operatorname{\leq 0}}_{\operatorname{qc}}(Y)$.
\end{lemma}

\begin{proof}
    Consider a perfect complex $G$ on $X$ satisifying $\overline{\langle G \rangle}^{(-\infty,0]} = D^{\operatorname{\leq 0}}_{\operatorname{qc}}(X)$. If $G$ is perfect, then $f^\ast G \cong \mathbf{L}f^\ast G$, Hence, $\mathbf{L}f^\ast G \in D^{\operatorname{\leq 0}}_{\operatorname{qc}}(Y)$, which implies $\overline{\langle \mathbf{L} f^\ast G \rangle}^{(-\infty,0]} \subseteq D^{\operatorname{\leq 0}}_{\operatorname{qc}}(Y)$. So, we need to check the reverse containment. However, this is equivalent to $D^{\geq 0}_{\operatorname{qc}}(Y) \supseteq (\langle \mathbf{L} f^\ast G \rangle^{(-\infty,0]})^{\perp}$. As $f$ is affine, we know that $f_\ast \colon \operatorname{Qcoh}(Y) \to \operatorname{Qcoh}(X)$ is exact and $f_\ast=\mathbf{R}f_\ast$ reflects isomorphisms. 
    It follows that $E\in D_{\operatorname{qc}}(Y)$ is in $D^{\leq 0}_{\operatorname{qc}}(Y)$ if and only if $f_\ast E\in D^{\leq 0}_{\operatorname{qc}}(X)$ (and similarly for $D^{\geq 0}$). Now, let $E\in D_{\operatorname{qc}}(Y)$ be such that $\operatorname{Hom}(\mathbf{L}f^\ast G,E[n])=0$ for all $n \geq 0$. Then, via adjunction, $\operatorname{Hom}(G,f_\ast E[n]) = 0$ for all $n \geq 0$. From our hypothesis on $G$, it follows that $f_\ast E\in D^{\leq 0}_{\operatorname{qc}}(X) = \overline{\langle G \rangle}^{(-\infty,0]}$. Yet, this implies $E\in D^{\geq 0}_{\operatorname{qc}}(Y)$, which is what we needed to show. 
\end{proof}

\begin{proposition}
    \label{prop:positive_relative_dimension}
    Let $f\colon Y \to X$ be a proper morphism of Noetherian algebraic spaces. If $D^{\operatorname{\leq 0}}_{\operatorname{qc}}(Y)$ is singly compact generated, then $f$ is finite. Equivalently, if $f$ is has positive relative dimension (i.e.\ has at least one fiber of positive Krull dimension), then $D^{\operatorname{\leq 0}}_{\operatorname{qc}}(Y)$ is not singly compact generated.
\end{proposition}

\begin{proof}
    Let $\operatorname{Spec}(k)\to X$ be a morphism from a field. Consider the fibered square,
    \begin{displaymath}
    \begin{tikzcd}
        {Y\times_X\operatorname{Spec}(k)} & {\operatorname{Spec}(k)} \\
        Y & X
        \arrow[from=1-1, to=1-2]
        \arrow[from=1-1, to=2-1]
        \arrow[from=1-2, to=2-2]
        \arrow["f", from=2-1, to=2-2]
    \end{tikzcd}
    \end{displaymath}
    Here, $Y\times_X \operatorname{Spec}(k)$ is a proper algebraic $k$-space. Note that $\operatorname{Spec}(k) \to X$ is affine (see e.g.\ \cite[\href{https://stacks.math.columbia.edu/tag/09TF}{Tag 09TF}]{stacks-project}) and so, by base change, $Y\times_X \operatorname{Spec}(k) \to Y$ is also affine. Using \Cref{lem:affine_base_change}, we know that $D^{\operatorname{\leq 0}}_{\operatorname{qc}}(Y)$ singly compactly generated implies $D^{\leq 0}_{\operatorname{qc}}(Y\times_X\operatorname{Spec}(k))$ is so too. However, \Cref{prop:proper_case} tells us that $Y\times_X \operatorname{Spec}(k)$ is an affine Artinian scheme. So, \cite[\href{https://stacks.math.columbia.edu/tag/06RW}{Tag 06RW}]{stacks-project} implies $f$ is locally quasi-finite. Moreover, by \cite[\href{https://stacks.math.columbia.edu/tag/0418}{Tag 0418}]{stacks-project} $f$ is representable (by schemes). Consequently, as proper quasi-finite morphisms of schemes are finite, $f$ is finite as desired.
\end{proof}

\begin{appendix}
\section{An alternative approach: Stability}
\label{sec:appendix_semistable}

Initially, we showed failure of singly compact generated standard aisles for smooth complex projective curves. Particularly, we used (semi)stability of vector bundles on said curves. See \cite{Huybrechts:2014,Macri/Schmidt:2017} for background. So, for sake of interest:


\begin{theorem}
    \label{thm:aisle_on_curve_single_object_fails}
    If $X$ is a smooth projective curve over $\mathbb{C}$, then $D^{\leq 0}_{\operatorname{qc}}(X)$ is not singly compactly generated.
\end{theorem}

\begin{proof}
    Assume the contrary; that is, we have an object $G\in D^b_{\operatorname{coh}}(X)$ satisfying $\overline{\langle G \rangle}^{(-\infty,0]} = D^{\leq 0}_{\operatorname{qc}}(X)$. To start, we make a few reductions regarding the appearance of $G$. As $D^b_{\operatorname{coh}}(X)$ is a hereditary category, $G\cong T \oplus F$ where $T$ is a direct sum of shifts of torsion sheaves and $F$ is a direct sum of shifts of torsion free sheaves. So, without loss of generality, we may impose $G$ be a coherent sheaf. However, torsion sheaves are iterated extensions of the structure sheaves for closed points, whereas the latter objects are cones of line bundles. Hence, if needed, we can assume $G$ is torsion free (i.e.\ $T\cong 0$). This ensures $G$ is a vector bundle on $X$. Consider the Harder--Narasimhan filtration of $G$, $0 :=G_0 \subseteq \cdots \subseteq G_n := G$. Consequently, $G$ is an iterated extension of the coherent sheaves $G_i/G_{i-1}$, which allows us to replace $G$ by the direct sum of the $G_i/G_{i-1}$. Here, each $G_i/G_{i-1}$ is a semistable vector bundle, and $\mu(G_1/G_0)> \cdots > \mu(G_n/G_{n-1})$. 
    
    Now, we can find the desired contradiction. If there is a semistable object $E$ in $\operatorname{coh}(X)$ such that $\mu(E)<\mu(G_n/G_{n-1})$, then $\operatorname{Hom}(G,E)=0$, which implies $E\in (\overline{\langle G \rangle}^{(-\infty,0]})^\perp$ (see e.g.\ \cite[Lemma 3.1]{AlonsoTarrio/Lopez/Salorio:2003}). Moreover, for any such $E$, we know that $E[1]\in D^{\geq 0}_{\operatorname{qc}}(X)$, which would be absurd. So, we would complete the proof if such an object existed. Fortunately, this is the case if one looks at the genus $g$ of $X$ 
    \begin{itemize}
        \item $g=0$: The stable vector bundles coincide with line bundles, and as each have the same slope. Hence, we can reduce to computing sheaf cohomology on $\mathbb{P}^1_{\mathbb{C}}$.
        \item $g=1$: Choose coprime integers $r,d$ such that $\frac{d}{r}<\mu(G_n/G_{n-1})$. There is a stable vector bundle of slope $\frac{d}{r}$. See \cite{Atiyah:1957, Polishchuk:2003}. 
        \item $g>1$: Choose integers $r,d$ such that $\frac{d}{r}<\mu(G_n/G_{n-1})$. Then, by \cite[Theorem 2.10]{Macri/Schmidt:2017}, the coarse moduli space parameterizing $S$-equivalence classes of semistable vector bundles on $X$ is nonempty. Loc.\ cit.\ is a summarization of \cite{Drezet/Narasimhan:1989, Newstead:2012, Seshadri:2982,LePotier:1997,Huybrechts/Lehn:2010}. 
        \qedhere
    \end{itemize}
\end{proof}

\end{appendix}

\bibliographystyle{alpha}
\bibliography{mainbib}

\end{document}